\documentclass{amsart}
\usepackage{amssymb,amsmath,amsthm}
\usepackage[dvips]{graphicx}

\let\oldmarginpar\marginpar
\renewcommand\marginpar[1]{\-\oldmarginpar[\raggedleft\footnotesize #1]%
{\raggedright\footnotesize #1}}

\begin{document}

\newtheorem{theorem}{Theorem}[section]
\newtheorem{corollary}[theorem]{Corollary}
\newtheorem{lemma}[theorem]{Lemma}
\newtheorem{proposition}[theorem]{Proposition}
\theoremstyle{definition}
\newtheorem{definition}[theorem]{Definition}
\theoremstyle{remark}
\newtheorem{remark}[theorem]{Remark}
\theoremstyle{definition}
\newtheorem{example}[theorem]{Example}

\def\rank{{\text{rank}\,}}

\numberwithin{equation}{section}

\title{Almost h-semi-slant Riemannian maps}

\author{Kwang-Soon Park}
\address{Department of Mathematical Sciences, Seoul National University, Seoul 151-747, Republic of Korea}
\email{parkksn@gmail.com}



\begin{abstract}
As a generalization of slant Riemannian maps \cite{S4}, semi-slant Riemannian maps \cite{P6}, almost h-slant submersions \cite{P},
and almost h-semi-slant submersions \cite{P3},
 we introduce the notion of almost h-semi-slant Riemannian maps from
almost quaternionic Hermitian manifolds to Riemannian manifolds. We investigate the integrability of distributions,
the harmonicity of such maps, the geometry of fibers,
 etc. We also deal with the condition for such maps to be totally geodesic and study some decomposition theorems.
Moreover, we give some examples.
\end{abstract}

\maketitle
\section{Introduction}\label{intro}
\addcontentsline{toc}{section}{Introduction}

Let $F$ be a $C^{\infty}$-map from a Riemannian manifold  $(M,g_M)$ to a Riemannian manifold $(N,g_N)$, according to the conditions on the map $F$,
the map $F$ is said to be a harmonic map \cite{BW}, a totally geodesic map \cite{BW}, an isometric immersion \cite{C}, a Riemannian submersion
(\cite{O}, \cite{W}, \cite{FIP}), a Riemannian map \cite{F}, etc. As we know,
if we consider the notions of an isometric immersion and a Riemannian submersion as the Riemannian
generalization of the notions of an  immersion and a submersion, then the notion of a Riemannian map may be the Riemannian generalization
of the notion of a subimmersion \cite{F}.

Given an isometric immersion $F$, its research is originated from Gauss' work, which studied surfaces in the Euclidean space $\mathbb{R}^3$
and there are lots of papers and books on this topic. In particular, B. Y. Chen introduced and studied some notions: generic submanifolds \cite{C0} and
slant submanifolds \cite{C1}. The notion of generic submanifolds contains the notions of real hypersurfaces, complex submanifolds, totally
real submanifolds, anti-holomorphic submanifolds, purely real submanifolds, and CR-submanifolds. And the notion of slant submanifolds has some similarities
with the notions of  slant submersions \cite{S}, semi-slant submersions \cite{PP}, almost h-slant submersions
\cite{P}, almost h-semi-slant submersions \cite{P3}, slant Riemannian maps \cite{S4}, and  semi-slant Riemannian maps \cite{P6}.
For the Riemannian submersion $F$, B. O'Neill \cite{O} and A. Gray \cite{G} firstly studied the map $F$.
Since then, there are several kinds of Riemannian submersions (\cite{P6}, references therein).
A. Fischer \cite{F} defined a Riemannian map $F$, which generalizes and unifies the notions of an isometric immersion, a Riemannian submersion, and an isometry. After that, there are a lot of papers on this topic. Moreover, B. Sahin introduced a slant Riemannian map \cite{S4} and the author defined
a semi-slant Riemannian map \cite{P6}.
As a generalization of slant Riemannian maps \cite{S4}, semi-slant Riemannian maps \cite{P6}, almost h-slant submersions \cite{P},
and almost h-semi-slant submersions \cite{P3}, we will define an almost h-semi-slant Riemannian map and a h-semi-slant Riemannian map.
And as we know, the quaternionic K\"{a}hler manifolds have applications in physics as the target
spaces for nonlinear $\sigma-$models with supersymmetry \cite{CMMS}.

The paper is organized as follows. In section 2 we recall some notions, which are needed for later use. In section 3
we define the notions of an almost h-semi-slant Riemannian map and a h-semi-slant Riemannian map and obtain some properties on them.
In section 4 using both an almost h-semi-slant Riemannian map and a h-semi-slant Riemannian map, we get some decomposition theorems.
In section 5  we obtain some examples.

\section{Preliminaries}\label{prelim}

Let $(M,E,g)$ be an almost quaternionic Hermitian manifold, where
$M$ is a $4m-$dimensional differentiable manifold, $g$ is a
Riemannian metric on $M$, and $E$ is a rank 3 subbundle of
$\text{End} (TM)$ such that for any point $p\in M$ with its some
neighborhood $U$, there exists a local basis $\{ J_1,J_2,J_3 \}$
of sections of $E$ on $U$ satisfying for all $\alpha\in \{ 1,2,3 \}$
\begin{equation}\label{eq: quat1}
J_{\alpha}^2=-id, \quad
J_{\alpha}J_{\alpha+1}=-J_{\alpha+1}J_{\alpha}=J_{\alpha+2},
\end{equation}
\begin{equation}\label{eq: Herm1}
g(J_{\alpha}X, J_{\alpha}Y)=g(X, Y)
\end{equation}
for all vector fields  $X, Y\in \Gamma(TM)$, where the indices are
taken from $\{ 1,2,3 \}$ modulo 3. The above basis $\{ J_1,J_2,J_3
\}$ is said to be a {\em quaternionic Hermitian basis}. We call
$(M,E,g)$ a {\em quaternionic K\"{a}hler manifold} if there exist
locally defined 1-forms $\omega_1, \omega_2, \omega_3$ such that
for $\alpha \in \{ 1,2,3 \}$
\begin{equation}\label{eq: quat2}
\nabla_X J_{\alpha} =
\omega_{\alpha+2}(X)J_{\alpha+1}-\omega_{\alpha+1}(X)J_{\alpha+2}
\end{equation}
for any vector field $X\in \Gamma(TM)$, where the indices are
taken from $\{ 1,2,3 \}$ modulo 3. If there exists a global
parallel quaternionic Hermitian basis $\{ J_1,J_2,J_3 \}$ of
sections of $E$ on $M$, then $(M, E, g )$ is said to be {\em
hyperk\"{a}hler}. Furthermore, we call $(J_1, J_2, J_3, g )$ a
{\em hyperk\"{a}hler structure} on $M$ and $g$ a {\em
hyperk\"{a}hler metric}.

Let $(M,g_M)$ and $(N,g_N)$ be Riemannian manifolds, where $M$, $N$ are $C^{\infty}$-manifolds and
$g_M$, $g_N$ are Riemannian metrics on $M$, $N$, respectively. Let $F : (M,g_M) \mapsto (N,g_N)$ be a $C^{\infty}$-map.
We call the map $F$ a {\em $C^{\infty}$-submersion} if $F$ is surjective and the differential $(F_*)_p$  has a maximal rank for any $p\in M$.
The map $F$ is said to be a {\em Riemannian submersion} \cite{O} if $F$ is a $C^{\infty}$-submersion and
$(F_*)_p : ((\ker (F_*)_p)^{\perp}, (g_M)_p) \mapsto (T_{F(p)} N, (g_N)_{F(p)})$ is a linear isometry for each $p\in M$,
where $(\ker (F_*)_p)^{\perp}$ is the orthogonal complement of the space $\ker (F_*)_p$ in the tangent space $T_p M$ of $M$ at $p$.
We call the map $F$ a {\em Riemannian map} \cite{F} if $(F_*)_p : ((\ker (F_*)_p)^{\perp}, (g_M)_p) \mapsto ((range F_*)_{F(p)}, (g_N)_{F(p)})$
is a linear isometry for each $p\in M$, where $(range F_*)_{F(p)} := (F_*)_p ((\ker (F_*)_p)^{\perp})$ for $p\in M$.

Let $(M,g_M,J)$ be an almost Hermitian manifold and $(N,g_N)$ a Riemannian manifold, where $J$ is an almost complex structure on $M$.
Let $F : (M,g_M,J) \mapsto (N,g_N)$ be a $C^{\infty}$-map. We call the map $F$ a {\em slant submersion} \cite{S} if $F$ is a Riemannian submersion
and the angle $\theta = \theta (X)$ between $JX$ and the space $\ker (F_*)_p$ is constant for nonzero $X\in \ker (F_*)_p$
and $p\in M$.

We call the angle $\theta$ a {\em slant angle}.

The map $F$ is said to be a {\em semi-slant submersion}
\cite{PP} if $F$ is a Riemannian submersion and there is a
distribution $\mathcal{D}_1\subset \ker F_*$ such that
$$
\ker F_* =\mathcal{D}_1\oplus \mathcal{D}_2, \
J(\mathcal{D}_1)=\mathcal{D}_1,
$$
and the angle $\theta=\theta(X)$ between $JX$ and the space
$(\mathcal{D}_2)_q$ is constant for nonzero $X\in
(\mathcal{D}_2)_q$ and $q\in M$, where $\mathcal{D}_2$ is the
orthogonal complement of $\mathcal{D}_1$ in $\ker F_*$.

We call the angle $\theta$ a {\em semi-slant angle}.

We call the map $F$ a {\em slant Riemannian map} \cite{S4} if $F$ is a Riemannian map and the angle $\theta = \theta (X)$ between $JX$
and the space $\ker (F_*)_p$ is constant for nonzero $X\in \ker (F_*)_p$ and $p\in M$.

We call the angle $\theta$ a {\em slant angle}.

The map $F$ is said to be a {\em semi-slant Riemannian map}
\cite{P6} if $F$ is a Riemannian map and there is a
distribution $\mathcal{D}_1\subset \ker F_*$ such that
$$
\ker F_* =\mathcal{D}_1\oplus \mathcal{D}_2, \
J(\mathcal{D}_1)=\mathcal{D}_1,
$$
and the angle $\theta=\theta(X)$ between $JX$ and the space
$(\mathcal{D}_2)_q$ is constant for nonzero $X\in
(\mathcal{D}_2)_q$ and $q\in M$, where $\mathcal{D}_2$ is the
orthogonal complement of $\mathcal{D}_1$ in $\ker F_*$.

We call the angle $\theta$ a {\em semi-slant angle}.

Let $(M, E, g_M)$ be an almost
quaternionic Hermitian manifold and $(N, g_N)$ a Riemannian
manifold. A Riemannian submersion $F : (M, E, g_M) \mapsto (N,
g_N)$ is said to be an {\em almost h-slant submersion} \cite{P} if given a
point $p\in M$ with its some neighborhood $U$, there exists a
 quaternionic Hermitian basis $\{ I,J,K \}$ of sections of
$E$ on $U$ such that for $R\in \{ I,J,K \}$ the angle
$\theta_R = \theta_R(X)$ between $RX$ and the space $\ker (F_*)_q$ is
constant for nonzero $X\in \ker (F_*)_q$ and $q\in U$.

We call such a basis $\{ I,J,K \}$ an {\em almost h-slant basis}
and the angles $\{ \theta_I,\theta_J,\theta_K \}$ {\em
almost h-slant angles}..

A Riemannian submersion $F : (M, E, g_M) \mapsto (N, g_N)$ is
called a {\em h-slant submersion} \cite{P} if given a point $p\in M$ with
its some neighborhood $U$, there exists a  quaternionic Hermitian
basis $\{ I,J,K \}$ of sections of $E$ on $U$ such that for $R\in
\{ I,J,K \}$ the angle $\theta_R = \theta_R(X)$ between $RX$ and the space
$\ker (F_*)_q$ is constant for nonzero $X\in \ker (F_*)_q$ and
$q\in U$, and $\theta=\theta_I=\theta_J=\theta_K$.

We call such a basis $\{ I,J,K \}$ a {\em h-slant basis} and the
angle $\theta$ a {\em h-slant angle}.

A Riemannian submersion $F : (M,E,g_M) \mapsto (N,g_N)$ is called a {\em h-semi-slant
submersion} \cite{P3} if given a point $p\in M$ with its some neighborhood
$U$, there exists a  quaternionic Hermitian basis $\{ I,J,K \}$ of
sections of $E$ on $U$ such that for any $R\in \{ I,J,K \}$, there
is a distribution $\mathcal{D}_1 \subset \ker F_*$ on $U$ such
that
$$
\ker F_* =\mathcal{D}_1\oplus \mathcal{D}_2, \
R(\mathcal{D}_1)=\mathcal{D}_1,
$$
and the angle $\theta_R=\theta_R(X)$ between $RX$ and the space
$(\mathcal{D}_2)_q$ is constant for nonzero $X\in
(\mathcal{D}_2)_q$ and $q\in U$, where $\mathcal{D}_2$ is the
orthogonal complement of $\mathcal{D}_1$ in $\ker F_*$.

We call such a basis $\{ I,J,K \}$ a {\em h-semi-slant basis} and
the angles $\{ \theta_I,\theta_J,\theta_K \}$ {\em h-semi-slant
angles}.

Furthermore, if we have
$$
\theta=\theta_I=\theta_J=\theta_K,
$$
then we call the map $F : (M,E,g_M)\mapsto (N,g_N)$ a {\em
strictly h-semi-slant submersion}, $\{ I,J,K \}$ a {\em strictly
h-semi-slant basis}, and the angle $\theta$ a {\em strictly
h-semi-slant angle}.

A Riemannian submersion $F : (M,E,g_M)\mapsto (N,g_N)$ is called an {\em almost h-semi-slant
submersion} \cite{P3} if given a point $p\in M$ with its some neighborhood
$U$, there exists a  quaternionic Hermitian basis $\{ I,J,K \}$ of
sections of $E$ on $U$ such that for each $R\in \{ I,J,K \}$,
there is a distribution $\mathcal{D}_1^R \subset \ker F_*$ on $U$
such that
$$
\ker F_* =\mathcal{D}_1^R\oplus \mathcal{D}_2^R, \
R(\mathcal{D}_1^R)=\mathcal{D}_1^R,
$$
and the angle $\theta_R=\theta_R(X)$ between $RX$ and the space
$(\mathcal{D}_2^R)_q$ is constant for nonzero $X\in
(\mathcal{D}_2^R)_q$ and $q\in U$, where $\mathcal{D}_2^R$ is the
orthogonal complement of $\mathcal{D}_1^R$ in $\ker F_*$.

We call such a basis $\{ I,J,K \}$ an {\em almost h-semi-slant
basis} and the angles $\{ \theta_I,\theta_J,\theta_K \}$ {\em
almost h-semi-slant angles}.

Let $(M, E_M, g_M)$ and $(N, E_N, g_N)$ be almost quaternionic
Hermitian manifolds. A map $F : M \mapsto N$ is called a {\em
$(E_M,E_N)-$holomorphic map} if given a point $x\in M$, for any
$J\in (E_M)_x$ there exists $J'\in (E_N)_{F(x)}$ such that
\begin{equation}\label{eq: holo1}
F_*\circ J=J'\circ F_*.
\end{equation}
A Riemannian submersion $F : M \mapsto N$ which is a
$(E_M,E_N)-$holomorphic map is called a {\em quaternionic
submersion}. Moreover, if $(M, E_M, g_M)$ is a quaternionic
K\"{a}hler manifold (or a hyperk\"{a}hler manifold), then we say
that $F$ is a {\em quaternionic K\"{a}hler submersion} (or a {\em
hyperk\"{a}hler submersion}) \cite{IMV}.

Let $F : (M, g_M) \mapsto (N, g_N)$ be
a $C^{\infty}$-map. The second fundamental form of $F$ is given by
\begin{equation}\label{eq: second}
(\nabla F_*)(X,Y) := \nabla^F _X F_* Y-F_* (\nabla _XY) \quad
\text{for} \ X,Y\in \Gamma(TM),
\end{equation}
where $\nabla^F$ is the pullback connection and we denote
conveniently by $\nabla$ the Levi-Civita connections of the
metrics $g_M$ and $g_N$ \cite{BW}. Recall that $F$ is said to be {\em
harmonic} if we have the tension field $\tau (F) := trace (\nabla F_*)=0$ and we call the map $F$ a {\em
totally geodesic} map if $(\nabla F_*)(X,Y)=0$ for $X,Y\in \Gamma
(TM)$ \cite{BW}.
Denote the range of $F_*$ by $range F_*$ as a subset of the pullback bundle $F^{-1} TN$. With its orthogonal complement $(range F_*)^{\perp}$
we obtain the following decomposition
\begin{equation}\label{eq: decom11}
F^{-1} TN = range F_* \oplus (range F_*)^{\perp}.
\end{equation}
Moreover, we have
\begin{equation}\label{eq: decom12}
TM = \ker F_* \oplus (\ker F_*)^{\perp}.
\end{equation}
Then we easily get

\begin{lemma}\label{lem: hori}
Let $F$ be a Riemannian map from a Riemannian manifold $(M,g_M)$ to a Riemannian manifold $(N,g_N)$. Then
\begin{equation}\label{eq: horiz11}
(\nabla F_*)(X,Y) \in \Gamma((range F_*)^{\perp}) \quad \text{for} \ X,Y\in \Gamma((\ker F_*)^{\perp}).
\end{equation}
\end{lemma}

\begin{lemma}
Let $F$ be a Riemannian map from a Riemannian manifold $(M,g_M)$ to a Riemannian manifold $(N,g_N)$. Then
the map $F$ satisfies a generalized eikonal equation \cite{F}
\begin{equation}\label{eq: eik1}
2e(F) = ||F_*||^2 = \rank F.
\end{equation}
\end{lemma}

As we can see, $||F_*||^2$ is continuous on $M$ and $\rank F$ is an integer-valued function on $M$ so that
$\rank F$ is locally constant. Hence, if $M$ is connected, then $\rank F$ is a constant function \cite{F}.
In \cite{F}, A. Fischer suggested that using (\ref{eq: eik1}), we may build a quantum model of nature. And if we can do it,
then there will be an interesting relationship between the mathematical side from Riemannian maps, harmonic maps, and Lagrangian
field theory and the physical side from Maxwell's equation, Schr\"{o}dinger's equation, and their proposed generalization.

\section{Almost h-semi-slant Riemannian maps}\label{semi}

\begin{definition}
Let $(M,E,g_M)$ be an almost quaternionic Hermitian manifold and
$(N,g_N)$ a Riemannian manifold. A Riemannian map $F :
(M,E,g_M)\mapsto (N,g_N)$ is called a {\em h-semi-slant
Riemannian map} if given a point $p\in M$ with its some neighborhood
$U$, there exists a  quaternionic Hermitian basis $\{ I,J,K \}$ of
sections of $E$ on $U$ such that for any $R\in \{ I,J,K \}$, there
is a distribution $\mathcal{D}_1 \subset \ker F_*$ on $U$ such
that
$$
\ker F_* =\mathcal{D}_1\oplus \mathcal{D}_2, \
R(\mathcal{D}_1)=\mathcal{D}_1,
$$
and the angle $\theta_R=\theta_R(X)$ between $RX$ and the space
$(\mathcal{D}_2)_q$ is constant for nonzero $X\in
(\mathcal{D}_2)_q$ and $q\in U$, where $\mathcal{D}_2$ is the
orthogonal complement of $\mathcal{D}_1$ in $\ker F_*$.
\end{definition}

We call such a basis $\{ I,J,K \}$ a {\em h-semi-slant basis} and
the angles $\{ \theta_I,\theta_J,\theta_K \}$ {\em h-semi-slant
angles}.

Furthermore, if we have
$$
\theta=\theta_I=\theta_J=\theta_K,
$$
then we call the map $F : (M,E,g_M)\mapsto (N,g_N)$ a {\em
strictly h-semi-slant Riemannian map}, $\{ I,J,K \}$ a {\em strictly
h-semi-slant basis}, and the angle $\theta$ a {\em strictly
h-semi-slant angle}.

\begin{definition}
Let $(M,E,g_M)$ be an almost quaternionic Hermitian manifold and
$(N,g_N)$ a Riemannian manifold. A Riemannian map $F :
(M,E,g_M)\mapsto (N,g_N)$ is called an {\em almost h-semi-slant
Riemannian map} if given a point $p\in M$ with its some neighborhood
$U$, there exists a  quaternionic Hermitian basis $\{ I,J,K \}$ of
sections of $E$ on $U$ such that for each $R\in \{ I,J,K \}$,
there is a distribution $\mathcal{D}_1^R \subset \ker F_*$ on $U$
such that
$$
\ker F_* =\mathcal{D}_1^R\oplus \mathcal{D}_2^R, \
R(\mathcal{D}_1^R)=\mathcal{D}_1^R,
$$
and the angle $\theta_R=\theta_R(X)$ between $RX$ and the space
$(\mathcal{D}_2^R)_q$ is constant for nonzero $X\in
(\mathcal{D}_2^R)_q$ and $q\in U$, where $\mathcal{D}_2^R$ is the
orthogonal complement of $\mathcal{D}_1^R$ in $\ker F_*$.
\end{definition}

We call such a basis $\{ I,J,K \}$ an {\em almost h-semi-slant
basis} and the angles $\{ \theta_I,\theta_J,\theta_K \}$ {\em
almost h-semi-slant angles}.

\noindent Let $F : (M,E,g_M)\mapsto (N,g_N)$ be an almost
h-semi-slant Riemannian map. Then given a point $p\in M$ with its some
neighborhood $U$, there exists a  quaternionic Hermitian basis $\{
I,J,K \}$ of sections of $E$ on $U$ such that for each $R\in \{
I,J,K \}$, there is a distribution $\mathcal{D}_1^R \subset \ker
F_*$ on $U$ such that
$$
\ker F_* =\mathcal{D}_1^R\oplus \mathcal{D}_2^R, \
R(\mathcal{D}_1^R)=\mathcal{D}_1^R,
$$
and the angle $\theta_R=\theta_R(X)$ between $RX$ and the space
$(\mathcal{D}_2^R)_q$ is constant for nonzero $X\in
(\mathcal{D}_2^R)_q$ and $q\in U$, where $\mathcal{D}_2^R$ is the
orthogonal complement of $\mathcal{D}_1^R$ in $\ker F_*$.

Then for $X\in \Gamma(\ker F_*)$, we write
\begin{equation}\label{eq: proj1}
X = P_RX+Q_RX,
\end{equation}
where $P_RX\in \Gamma(\mathcal{D}_1^R)$ and $Q_RX\in
\Gamma(\mathcal{D}_2^R)$.

For $X\in \Gamma(\ker F_*)$, we obtain
\begin{equation}\label{eq: proj2}
RX = \phi_R X+\omega_R X,
\end{equation}
where $\phi_R X\in \Gamma(\ker F_*)$ and $\omega_R X\in
\Gamma((\ker F_*)^{\perp})$.

For $Z\in \Gamma((\ker F_*)^{\perp})$, we have
\begin{equation}\label{eq: proj3}
RZ = B_RZ+C_RZ,
\end{equation}
where $B_RZ\in \Gamma(\ker F_*)$ and $C_RZ\in \Gamma((\ker
F_*)^{\perp})$.

For $U\in \Gamma(TM)$, we get
\begin{equation}\label{eq: proj4}
U = \mathcal{V}U+\mathcal{H}U,
\end{equation}
where $\mathcal{V}U\in \Gamma(\ker F_*)$ and $\mathcal{H}U\in
\Gamma((\ker F_*)^{\perp})$.

For $W\in \Gamma(F^{-1}TN)$, we write
\begin{equation}\label{eq: proj5}
W = \bar{P}W+\bar{Q}W,
\end{equation}
where $\bar{P}W\in \Gamma(range F_*)$ and $\bar{Q}W\in \Gamma((range F_*)^{\perp})$.

Then
\begin{equation}\label{eq: proj6}
(\ker F_*)^{\perp} = \omega_R \mathcal{D}_2^R \oplus \mu_R,
\end{equation}
where $\mu_R$ is the orthogonal complement of $\omega_R
\mathcal{D}_2^R$ in $(\ker F_*)^{\perp}$ and is invariant  under
$R$.

Furthermore,
{\setlength\arraycolsep{2pt}
\begin{eqnarray*}
& & \phi_R \mathcal{D}_1^R = \mathcal{D}_1^R, \ \omega_R
\mathcal{D}_1^R = 0, \ \phi_R \mathcal{D}_2^R \subset
\mathcal{D}_2^R, \
B_R((\ker F_*)^{\perp}) = \mathcal{D}_2^R        \\
& & \phi_R^2+B_R\omega_R = -id, \ C_R^2+\omega_R B_R = -id, \
\omega_R \phi_R +C_R\omega_R = 0, \ B_RC_R+\phi_R B_R = 0.
\end{eqnarray*}}
Define the tensors $\mathcal{T}$ and $\mathcal{A}$ by
\begin{eqnarray}
  \mathcal{A}_E F & = & \mathcal{H}\nabla_{\mathcal{H}E} \mathcal{V}F+\mathcal{V}\nabla_{\mathcal{H}E} \mathcal{H}F \label{eq: oten1} \\
   \mathcal{T}_E F & = & \mathcal{H}\nabla_{\mathcal{V}E} \mathcal{V}F+\mathcal{V}\nabla_{\mathcal{V}E} \mathcal{H}F \label{eq: oten2}
\end{eqnarray}
for $E, F\in \Gamma(TM)$, where $\nabla$ is the Levi-Civita
connection of $g_M$.

For $X,Y\in \Gamma(\ker F_*)$, define
\begin{equation}\label{eq: vert}
\widehat{\nabla}_X Y := \mathcal{V}\nabla_X Y
\end{equation}
\begin{eqnarray}
(\nabla_X \phi)Y & := & \widehat{\nabla}_X \phi Y-\phi
\widehat{\nabla}_X Y \label{eq: vertc}  \\
(\nabla_X \omega)Y & := & \mathcal{H}\nabla_X \omega
Y-\omega\widehat{\nabla}_X Y. \label{eq: horizc}
\end{eqnarray}
Then we easily obtain

\begin{lemma}
Let $F$ be an almost h-semi-slant Riemannian map from a
hyperk\"{a}hler manifold $(M,I,J,K,g_M)$ to a Riemannian
manifold $(N, g_N)$ such that $(I,J,K)$ is an almost h-semi-slant
basis. Then we get

\begin{enumerate}
\item
\begin{align*}
  &\widehat{\nabla}_X \phi_R Y+\mathcal{T}_X \omega_R Y = \phi_R\widehat{\nabla}_X Y+B_R\mathcal{T}_X Y    \\
  &\mathcal{T}_X \phi_R Y+\mathcal{H}\nabla_X \omega_R Y =
  \omega_R\widehat{\nabla}_X Y+C_R\mathcal{T}_X Y
\end{align*}
for $X,Y\in \Gamma(\ker F_*)$ and $R\in \{ I,J,K \}$.
\item
\begin{align*}
  &\mathcal{V}\nabla_Z B_RW+\mathcal{A}_Z C_RW = \phi_R\mathcal{A}_Z W+B_R\mathcal{H}\nabla_Z W    \\
  &\mathcal{A}_Z B_RW+\mathcal{H}\nabla_Z C_RW = \omega_R\mathcal{A}_Z
  W+C_R\mathcal{H}\nabla_Z W
\end{align*}
for $Z,W\in \Gamma((\ker F_*)^{\perp})$ and $R\in \{ I,J,K \}$.
\item
\begin{align*}
  &\widehat{\nabla}_X B_RZ+\mathcal{T}_X C_RZ = \phi_R\mathcal{T}_X Z+B_R\mathcal{H}\nabla_X Z    \\
  &\mathcal{T}_X B_RZ+\mathcal{H}\nabla_X C_RZ = \omega_R\mathcal{T}_X Z+C_R\mathcal{H}\nabla_X Z  \\
  &\mathcal{V}\nabla_Z \phi_R X + \mathcal{A}_Z \omega_R X = \phi_R \mathcal{V}\nabla_Z X + B_R\mathcal{A}_Z X     \\
  &\mathcal{A}_Z \phi_R X + \mathcal{H}\nabla_Z \omega_R X = \omega_R\mathcal{V}\nabla_Z X + C_R\mathcal{A}_Z X
\end{align*}
for $X\in \Gamma(\ker F_*)$, $Z\in \Gamma((\ker F_*)^{\perp})$,
and $R\in \{ I,J,K \}$.
\end{enumerate}
\end{lemma}

Using the h-semi-slant Riemannian map $F$, we have

\begin{theorem}
Let $F$ be a h-semi-slant Riemannian map from a hyperk\"{a}hler
manifold $(M,I,J,K,g_M)$ to a Riemannian manifold $(N, g_N)$
such that $(I,J,K)$ is a h-semi-slant basis. Then the following
conditions are equivalent:

a) the complex distribution $\mathcal{D}_1$ is integrable.

b) $ Q_I(\widehat{\nabla}_X \phi_I Y - \widehat{\nabla}_Y \phi_I
X) = 0 \ \text{and} \ \mathcal{T}_X \phi_I Y = \mathcal{T}_Y
\phi_I X \quad \text{for} \ X,Y\in \Gamma(\mathcal{D}_1).$

c) $ Q_J(\widehat{\nabla}_X \phi_J Y - \widehat{\nabla}_Y \phi_J
X) = 0 \ \text{and} \ \mathcal{T}_X \phi_J Y = \mathcal{T}_Y
\phi_J X \quad \text{for} \ X,Y\in \Gamma(\mathcal{D}_1).$

d) $ Q_K(\widehat{\nabla}_X \phi_K Y - \widehat{\nabla}_Y \phi_K
X) = 0 \ \text{and} \ \mathcal{T}_X \phi_K Y = \mathcal{T}_Y
\phi_K X \quad \text{for} \ X,Y\in \Gamma(\mathcal{D}_1).$
\end{theorem}

\begin{proof}
Given $X,Y\in \Gamma(\mathcal{D}_1)$ and $R\in \{ I,J,K \}$, we obtain
\begin{align*}
R[X,Y]&= R(\nabla_X Y-\nabla_Y X)=\nabla_X RY-\nabla_Y RX    \\
      &= \widehat{\nabla}_X \phi_R Y-\widehat{\nabla}_Y \phi_R X +
      \mathcal{T}_X \phi_R Y-\mathcal{T}_Y \phi_R X.
\end{align*}
Hence, we get
$$
a) \Leftrightarrow b), \quad a) \Leftrightarrow c), \quad a)
\Leftrightarrow d).
$$
Therefore, the result follows.
\end{proof}

\begin{theorem}
Let $F$ be a h-semi-slant Riemannian map from a hyperk\"{a}hler
manifold $(M,I,J,K,g_M)$ to a Riemannian manifold $(N, g_N)$
such that $(I,J,K)$ is a h-semi-slant basis. Then the following
conditions are equivalent:

a) the slant distribution $\mathcal{D}_2$ is integrable.

b) $P_I(\widehat{\nabla}_X \phi_I Y-\widehat{\nabla}_Y \phi_I
X+\mathcal{T}_X \omega_I Y-\mathcal{T}_Y \omega_I X) = 0 \quad
\text{for} \ X,Y\in \Gamma(\mathcal{D}_2).$

c) $P_J(\widehat{\nabla}_X \phi_J Y-\widehat{\nabla}_Y \phi_J
X+\mathcal{T}_X \omega_J Y-\mathcal{T}_Y \omega_J X) = 0 \quad
\text{for} \ X,Y\in \Gamma(\mathcal{D}_2).$

d) $P_K(\widehat{\nabla}_X \phi_K Y-\widehat{\nabla}_Y \phi_K
X+\mathcal{T}_X \omega_K Y-\mathcal{T}_Y \omega_K X) = 0 \quad
\text{for} \ X,Y\in \Gamma(\mathcal{D}_2).$
\end{theorem}

\begin{proof}
Given $X,Y\in \Gamma(\mathcal{D}_2)$, $Z\in \Gamma(\mathcal{D}_1)$,
and $R\in \{ I,J,K \}$,  we obtain
\begin{align*}
g_M (R[X,Y], Z)&= g_M (\nabla_X RY-\nabla_Y RX, Z)    \\
      &= g_M (\widehat{\nabla}_X \phi_R Y+\mathcal{T}_X \phi_R Y+\mathcal{T}_X \omega_R Y
      +\mathcal{H}\nabla_X \omega_R Y-\widehat{\nabla}_Y\phi_R X   \\
      &\ \ \ -\mathcal{T}_Y \phi_R X  -\mathcal{T}_Y \omega_R X-\mathcal{H}\nabla_Y \omega_R X, Z)    \\
      &= g_M (\widehat{\nabla}_X \phi_R Y+\mathcal{T}_X \omega_R Y-\widehat{\nabla}_Y\phi_R X
      -\mathcal{T}_Y \omega_R X, Z).
\end{align*}
Since $[X,Y]\in \Gamma(\ker F_*)$, we get
$$
a) \Leftrightarrow b), \quad a) \Leftrightarrow c), \quad a)
\Leftrightarrow d).
$$
Therefore, we have the result.
\end{proof}

In the same way with the proof of Proposition 2.6 in \cite{P3}, we can show

\begin{proposition}\label{prop:slant}
Let $F$ be an almost h-semi-slant Riemannian map from an almost
quaternionic Hermitian manifold $(M,E,g_M)$ to a Riemannian
manifold $(N,g_N)$. Then we have
\begin{equation}\label{eq: slant11}
\phi_R^2 X = -\cos^2 \theta_R X \quad \text{for} \ X\in
\Gamma(\mathcal{D}_2^R) \ \text{and} \ R\in \{ I,J,K \},
\end{equation}
where $\{ I,J,K \}$ is an almost h-semi-slant basis with the
almost h-semi-slant angles $\{ \theta_I,\theta_J,\theta_K \}$.
\end{proposition}

\begin{remark}
In particular, it is easy to obtain that the converse of Proposition
\ref{prop:slant} is also true.
\end{remark}

Since
\begin{eqnarray*}
g_M(\phi_R X, \phi_R Y) & = & \cos^2 \theta_R \ g_M(X, Y)  \\
g_M(\omega_R X, \omega_R Y) & = & \sin^2 \theta_R \ g_M(X, Y)
\end{eqnarray*}
for $ X,Y\in \Gamma(\mathcal{D}_2^R)$, if $\displaystyle{\theta_R\in (0, \frac{\pi}{2})}$, then we can locally choose an orthonormal frame
  $\{ f_1, \sec \theta_R \phi_R f_1, \cdots, f_s, \sec \theta_R \phi_R f_s \}$ of $\mathcal{D}_2^R$.

Using (\ref{eq: slant11}), in a similar way with Lemma 3.5 in \cite{P6}, we can obtain

\begin{lemma}\label{lem: angle}
Let $F$ be an almost h-semi-slant Riemannian map from a
hyperk\"{a}hler manifold $(M,I,J,K,g_M)$ to a Riemannian
manifold $(N, g_N)$ such that $(I,J,K)$ is an almost h-semi-slant
basis with the almost h-semi-slant angles $\{ \theta_I,\theta_J,\theta_K \}$.
If the tensor $\omega_R$ is parallel, then we have
\begin{equation}\label{eq: angle4}
\mathcal{T}_{\phi_R X} \phi_R X = -\cos^2 \theta_R \cdot \mathcal{T}_X X \quad \text{for} \ X\in \Gamma(\mathcal{D}_2^R),
\end{equation}
where $R \in \{ I,J,K \}$.
\end{lemma}

Given an almost h-semi-slant Riemannian map $F$ from an almost quaternionic Hermitian
manifold $(M,E,g_M)$ to a Riemannian manifold $(N,g_N)$,  for some $R \in \{ I,J,K \}$ with
 $\displaystyle{\theta_R\in [0, \frac{\pi}{2})}$, we can define an endomorphism
$\widehat{R}$ of $\ker F_*$ by
$$
\widehat{R} := RP_R+\sec \theta_R \phi_R Q_R.
$$
Then
\begin{equation} \label{eq: compst}
{\widehat{R}}^2 = -id \quad \text{on} \ \ker F_*.
\end{equation}
Note that the distribution $\ker F_*$ is integrable. But its dimension may be odd. With the endomorphism $\widehat{R}$ we get

\begin{theorem}
Let $F$ be an almost h-semi-slant Riemannian map from an almost
quaternionic Hermitian manifold $(M,E,g_M)$ to a Riemannian manifold $(N,g_N)$ with
the almost h-semi-slant angles $\{ \theta_I,\theta_J,\theta_K \}$ not all $\displaystyle{\frac{\pi}{2}}$.
Then the fibers $F^{-1}(x)$ are even dimensional submanifolds of $M$ for $x\in M$.
\end{theorem}

Now, we consider the harmonicity of such maps. Let $F$ be a $C^{\infty}$-map  from a Riemannian manifold $(M,g_M)$ to a Riemannian manifold
$(N,g_N)$. We can canonically define a function $e(F) : M\mapsto [0, \infty)$ given by
\begin{equation} \label{eq: energ1}
e(F)(x) := \frac{1}{2} |(F_*)_x|^2, \quad x\in M,
\end{equation}
where $|(F_*)_x|$ denotes the Hilbert-Schmidt norm of $(F_*)_x$ \cite{BW}. Then the function $e(F)$ is said to be the {\em energy density} of $F$.
Let $D$ be a compact domain of $M$, i.e., $D$ is the compact closure $\bar{U}$ of a non-empty connected open subset $U$ of $M$.
The {\em energy integral} of $F$ over $D$ is the integral of its energy density:
\begin{equation} \label{eq: energ2}
E(F;D) := \int_D e(F) v_{g_M} = \frac{1}{2} \int_D |F_*|^2 v_{g_M},
\end{equation}
where $v_{g_M}$ is the volume form on $(M,g_M)$.
Let $C^{\infty}(M,N)$ denote the space of all $C^{\infty}$-maps from $M$ to $N$. A $C^{\infty}$-map $F : M\mapsto N$ is said to be
{\em harmonic} if it is a critical point of the energy functional $E(\ ;D) : C^{\infty}(M,N)\mapsto \mathbb{R}$ for any compact domain $D\subset M$.
By the result of J. Eells and J. Sampson \cite{ES}, we see that the map $F$ is harmonic if and only if the tension field $\tau (F) := trace \nabla F_* = 0$.

\begin{theorem}
Let $F$ be an almost h-semi-slant Riemannian map from a
hyperk\"{a}hler manifold $(M,I,J,K,g_M)$ to a Riemannian
manifold $(N, g_N)$ such that $(I,J,K)$ is an almost h-semi-slant basis.
Assume that $\widetilde{H} = 0$, where $\widetilde{H}$ denotes the mean curvature vector field of $range F_*$.
Then each of the following conditions implies that $F$ is harmonic.

a) $\mathcal{D}_1^I$ is integrable and $trace (\nabla F_*) = 0$ on $\mathcal{D}_2^I$.

b) $\mathcal{D}_1^J$ is integrable and $trace (\nabla F_*) = 0$ on $\mathcal{D}_2^J$.

c) $\mathcal{D}_1^K$ is integrable and $trace (\nabla F_*) = 0$ on $\mathcal{D}_2^K$.
\end{theorem}

\begin{proof}
Using Lemma \ref{lem: hori}, we get $trace \nabla F_*|_{\ker F_*}\in \Gamma(range F_*)$ and $trace \nabla F_*|_{(\ker F_*)^{\perp}}$
$\in \Gamma((range F_*)^{\perp})$ so that from (\ref{eq: decom12}), we have
$$
trace (\nabla F_*) = 0 \quad \Leftrightarrow \quad trace \nabla F_*|_{\ker F_*} = 0 \ \text{and} \ trace \nabla F_*|_{(\ker F_*)^{\perp}} = 0.
$$
Moreover, we easily obtain
$$
trace \nabla F_*|_{(\ker F_*)^{\perp}} = l\widetilde{H} \quad \text{for} \ l := \dim (\ker F_*)^{\perp}
$$
so that
$$
trace \nabla F_*|_{(\ker F_*)^{\perp}} = 0 \quad \Leftrightarrow \quad \widetilde{H} = 0.
$$
Given $R \in \{ I,J,K \}$, since $\mathcal{D}_1^R = R(\mathcal{D}_1^R)$,
 we can choose locally an orthonormal frame $\{ e_1, Re_1,$ $\cdots, e_k, Re_k \}$ of $\mathcal{D}_1^R$ so that
\begin{eqnarray*}
(\nabla F_*)(Re_i,Re_i) & = & -F_* \nabla_{Re_i} Re_i = -F_* R(\nabla_{e_i} Re_i+[Re_i, e_i])    \\
                        & = & F_* \nabla_{e_i} e_i-F_* R[Re_i, e_i] = -(\nabla F_*)(e_i,e_i)-F_* R[Re_i, e_i]
\end{eqnarray*}
for $1\leq i\leq k$.

Thus,
$$
\mathcal{D}_1^R \ \text{is integrable} \quad \Rightarrow \quad  trace \nabla F_*|_{\mathcal{D}_1^R} = 0.
$$
Since $\mathcal{D}_2^R$ is the orthogonal complement of $\mathcal{D}_1^R$ in $\ker F_*$, we have the result.
\end{proof}

Using Lemma \ref {lem: angle}, we obtain

\begin{corollary}
Let $F$ be an almost h-semi-slant Riemannian map from a
hyperk\"{a}hler manifold $(M,I,J,K,g_M)$ to a Riemannian
manifold $(N, g_N)$ such that $(I,J,K)$ is an almost h-semi-slant basis with
the almost h-semi-slant angles $\{ \theta_I,\theta_J,$ $\theta_K \}$.
Assume that $\widetilde{H} = 0$.
Then each of the following conditions implies that $F$ is harmonic.

a) $\mathcal{D}_1^I$ is integrable, the tensor $\omega_I$ is parallel, and $\theta_I\in [0,\frac{\pi}{2})$.

b) $\mathcal{D}_1^J$ is integrable, the tensor $\omega_J$ is parallel, and $\theta_J\in [0,\frac{\pi}{2})$.

c) $\mathcal{D}_1^K$ is integrable, the tensor $\omega_K$ is parallel, and $\theta_K\in [0,\frac{\pi}{2})$.
\end{corollary}

We now investigate the condition for such a map $F$ to be totally geodesic.

\begin{theorem}
Let $F$ be an almost h-semi-slant Riemannian map from a
hyperk\"{a}hler manifold $(M,I,J,K,g_M)$ to a Riemannian
manifold $(N, g_N)$ such that $(I,J,K)$ is an almost h-semi-slant
basis. Assume that $\bar{Q}(\nabla_{Z_1}^F F_* Z_2) = 0$ for $Z_1,Z_2\in \Gamma((\ker F_*)^{\perp})$.
Then the following conditions are equivalent:

a) $F$ is a totally geodesic map.

b)
\begin{align*}
  &\omega_I(\widehat{\nabla}_X \phi_I Y+\mathcal{T}_X \omega_I Y)+C_I(\mathcal{T}_X \phi_I Y
    +\mathcal{H}\nabla_X \omega_I Y) = 0   \\
  &\omega_I(\widehat{\nabla}_X B_IZ+\mathcal{T}_X C_IZ)+C_I(\mathcal{T}_X B_IZ
    +\mathcal{H}\nabla_X C_IZ) = 0
\end{align*}
for $X,Y\in \Gamma(\ker F_*)$ and $Z\in \Gamma((\ker
F_*)^{\perp})$.

c)
\begin{align*}
  &\omega_J(\widehat{\nabla}_X \phi_J Y+\mathcal{T}_X \omega_J Y)+C_J(\mathcal{T}_X \phi_J Y
    +\mathcal{H}\nabla_X \omega_J Y) = 0   \\
  &\omega_J(\widehat{\nabla}_X B_JZ+\mathcal{T}_X C_JZ)+C_J(\mathcal{T}_X B_JZ
    +\mathcal{H}\nabla_X C_JZ) = 0
\end{align*}
for $X,Y\in \Gamma(\ker F_*)$ and $Z\in \Gamma((\ker
F_*)^{\perp})$.

d)
\begin{align*}
  &\omega_K(\widehat{\nabla}_X \phi_K Y+\mathcal{T}_X \omega_K Y)+C_K(\mathcal{T}_X \phi_K Y
    +\mathcal{H}\nabla_X \omega_K Y) = 0   \\
  &\omega_K(\widehat{\nabla}_X B_KZ+\mathcal{T}_X C_KZ)+C_K(\mathcal{T}_X B_KZ
    +\mathcal{H}\nabla_X C_KZ) = 0
\end{align*}
for $X,Y\in \Gamma(\ker F_*)$ and $Z\in \Gamma((\ker
F_*)^{\perp})$.
\end{theorem}

\begin{proof}
If $Z_1,Z_2\in \Gamma((\ker F_*)^{\perp})$, then by Lemma \ref {lem: hori}, we get
$$
(\nabla F_*)(Z_1,Z_2)=0 \quad \Leftrightarrow \quad \bar{Q}((\nabla F_*)(Z_1,Z_2)) = \bar{Q}(\nabla_{Z_1}^F F_* Z_2) = 0.
$$
For $X,Y\in \Gamma(\ker F_*)$, we obtain
\begin{align*}
(\nabla F_*)(X,Y)&= -F_* (\nabla_X Y) = F_* (I\nabla_X (\phi_I Y+\omega_I Y))   \\
          &= F_* (\phi_I \widehat{\nabla}_X \phi_I Y+\omega_I\widehat{\nabla}_X \phi_I Y+B_I\mathcal{T}_X \phi_I Y
             +C_I\mathcal{T}_X \phi_I Y+\phi_I \mathcal{T}_X \omega_I Y   \\
          & \ \ \ +\omega_I \mathcal{T}_X \omega_I Y +B_I\mathcal{H}\nabla_X \omega_I Y
          +C_I\mathcal{H}\nabla_X \omega_I Y).
\end{align*}
Thus,
$$
(\nabla F_*)(X,Y) = 0 \Leftrightarrow \omega_I(\widehat{\nabla}_X
\phi_I Y+\mathcal{T}_X \omega_I Y)+C_I(\mathcal{T}_X \phi_I Y
    +\mathcal{H}\nabla_X \omega_I Y) = 0.
$$
Given $X\in \Gamma(\ker F_*)$ and $Z\in \Gamma((\ker F_*)^{\perp})$,
since $(\nabla F_*)(X,Z)=(\nabla F_*)(Z,X)$, it is sufficient to
consider the following case:
\begin{align*}
(\nabla F_*)(X,Z)&= -F_* (\nabla_X Z) = F_* (I\nabla_X (B_IZ+C_IZ))   \\
          &= F_* (\phi_I \widehat{\nabla}_X B_IZ+\omega_I\widehat{\nabla}_X B_IZ+B_I\mathcal{T}_X B_IZ
             +C_I\mathcal{T}_X B_IZ+\phi_I \mathcal{T}_X C_IZ   \\
          & \ \ \ +\omega_I \mathcal{T}_X C_IZ +B_I\mathcal{H}\nabla_X C_IZ+C_I\mathcal{H}\nabla_X C_IZ)
\end{align*}
so that
$$
(\nabla F_*)(X,Z) = 0 \Leftrightarrow \omega_I(\widehat{\nabla}_X
B_IZ+\mathcal{T}_X C_IZ)+C_I(\mathcal{T}_X B_IZ
    +\mathcal{H}\nabla_X C_IZ) = 0.
$$
Hence,
$$
a) \Leftrightarrow b).
$$
Similarly,
$$
\quad a) \Leftrightarrow c) \quad \text{and} \quad a)
\Leftrightarrow d).
$$
Therefore, we get the result.
\end{proof}

Let $F : (M,g_M)\mapsto (N,g_N)$ be a Riemannian map. The map
$F$ is called a Riemannian map {\em with totally umbilical
fibers} if
\begin{equation}\label{eq: umbil}
\mathcal{T}_X Y = g_M (X, Y)H \quad \text{for} \ X,Y\in \Gamma(\ker F_*),
\end{equation}
where $H$ is the mean curvature vector field of the fiber.

In a similar way with Lemma 2.17 in \cite {P3}, we have

\begin{lemma}\label{lem: umb}
Let $F$ be an almost h-semi-slant Riemannian map with totally
umbilical fibers from a hyperk\"{a}hler manifold $(M,I,J,K,g_M)$
 to a Riemannian manifold $(N, g_N)$ such that $(I,J,K)$ is an
almost h-semi-slant basis. Then we get
\begin{equation}\label{eq: umbil2}
H\in \Gamma(\omega_R \mathcal{D}_2^R) \quad \text{for} \ R\in \{ I,J,K \}.
\end{equation}
\end{lemma}

Using Lemma \ref {lem: umb}, we obtain

\begin{corollary}
Let $F$ be an almost h-semi-slant Riemannian map with totally
umbilical fibers from a hyperk\"{a}hler manifold $(M,I,J,K,g_M)$ to a Riemannian
manifold $(N, g_N)$ such that $(I,J,K)$ is an almost h-semi-slant basis with
the almost h-semi-slant angles $\{ \theta_I,\theta_J,\theta_K \}$.
Assume that $\theta_R = 0$ for some $R \in \{ I,J,K \}$.
Then the fibers of $F$ are minimal submanifolds of $M$.
\end{corollary}

\section{Decomposition theorems}\label{decom}

Let $(M,g_M)$ be a Riemannian manifold and $\mathcal{D}$ a ($C^{\infty}$-) distribution  on $M$.
The distribution $\mathcal{D}$ is said to be
{\em autoparallel} (or a {\em totally geodesic foliation}) if $\nabla_X Y\in \Gamma(\mathcal{D})$ for $X,Y\in \Gamma(\mathcal{D})$.
Given an autoparallel distribution $\mathcal{D}$ on $M$, it is easy to see that $\mathcal{D}$ is integrable and its leaves are totally geodesic in $M$.
Moreover, we call the distribution $\mathcal{D}$ {\em parallel} if $\nabla_Z Y\in \Gamma(\mathcal{D})$ for $Y\in \Gamma(\mathcal{D})$ and $Z\in \Gamma(TM)$.
Given a parallel distribution $\mathcal{D}$ on $M$, we easily obtain that its orthogonal complementary distribution $\mathcal{D}^{\perp}$ is also parallel.
In this case, $M$ is locally a Riemannian product manifold of the leaves of $\mathcal{D}$ and $\mathcal{D}^{\perp}$.
We can also obtain that if the distributions $\mathcal{D}$ and $\mathcal{D}^{\perp}$ are simultaneously autoparallel, then
they are also parallel. Using this fact, we have

\begin{theorem}
Let $F$ be an almost h-semi-slant Riemannian map from a
hyperk\"{a}hler manifold $(M,I,J,K,g_M)$ to a Riemannian
manifold $(N, g_N)$ such that $(I,J,K)$ is an almost h-semi-slant
basis. Then the following conditions are equivalent:

a) $(M,g_M)$ is locally a Riemannian product manifold of the leaves of $\ker F_*$ and $(\ker F_*)^{\perp}$

b)
$$
\omega_I (\widehat{\nabla}_X \phi_I Y+\mathcal{T}_X \omega_I
Y)+C_I(\mathcal{T}_X \phi_I Y+\mathcal{H}\nabla_X \omega_I Y) = 0
\quad \text{for} \ X,Y\in \Gamma(\ker F_*),
$$
$$
\phi_I(\mathcal{V}{\nabla}_Z B_IW+\mathcal{A}_Z C_IW)+B_I(\mathcal{A}_Z
B_IW+\mathcal{H}\nabla_Z C_IW) = 0  \quad
\text{for} \ Z,W\in \Gamma((\ker F_*)^{\perp}).
$$

c)
$$
\omega_J (\widehat{\nabla}_X \phi_J Y+\mathcal{T}_X \omega_J
Y)+C_J(\mathcal{T}_X \phi_J Y+\mathcal{H}\nabla_X \omega_J Y) = 0
\quad \text{for} \ X,Y\in \Gamma(\ker F_*),
$$
$$
\phi_J(\mathcal{V}{\nabla}_Z B_JW+\mathcal{A}_Z C_JW)+B_J(\mathcal{A}_Z
B_JW+\mathcal{H}\nabla_Z C_JW) = 0 \quad
\text{for} \ Z,W\in \Gamma((\ker F_*)^{\perp}).
$$

d)
$$
\omega_K (\widehat{\nabla}_X \phi_K Y+\mathcal{T}_X \omega_K
Y)+C_K(\mathcal{T}_X \phi_K Y+\mathcal{H}\nabla_X \omega_K Y) = 0
\quad \text{for} \ X,Y\in \Gamma(\ker F_*),
$$
$$
\phi_K(\mathcal{V}{\nabla}_Z B_KW+\mathcal{A}_Z C_KW)+B_K(\mathcal{A}_Z
B_KW+\mathcal{H}\nabla_Z C_KW) = 0 \quad
\text{for} \ Z,W\in \Gamma((\ker F_*)^{\perp}).
$$
\end{theorem}

\begin{proof}
Given $R\in \{ I,J,K \}$, for $X,Y\in \Gamma(\ker F_*)$, we get
\begin{align*}
\nabla_X Y&= -R\nabla_X RY= -R(\widehat{\nabla}_X \phi_R Y+\mathcal{T}_X \phi_R Y+\mathcal{T}_X \omega_R Y
             +\mathcal{H}\nabla_X \omega_R Y)   \\
          &= -(\phi_R \widehat{\nabla}_X \phi_R Y+\omega_R\widehat{\nabla}_X \phi_R Y+B_R\mathcal{T}_X \phi_R Y
             +C_R\mathcal{T}_X \phi_R Y+\phi_R \mathcal{T}_X \omega_R Y   \\
          & \ \ \ +\omega_R \mathcal{T}_X \omega_R Y  +B_R\mathcal{H}\nabla_X \omega_R Y+C_R\mathcal{H}\nabla_X \omega_R Y).
\end{align*}
Thus,
$$
\nabla_X Y\in \Gamma(\ker F_*) \Leftrightarrow
\omega_R(\widehat{\nabla}_X \phi_R Y+\mathcal{T}_X \omega_R
Y)+C_R(\mathcal{T}_X \phi_R Y+\mathcal{H}\nabla_X \omega_R Y) =0.
$$
For $Z,W\in \Gamma((\ker F_*)^{\perp})$, we have
\begin{align*}
\nabla_Z W & = -R\nabla_Z RW = -R(\mathcal{V}\nabla_Z B_RW+\mathcal{A}_Z
B_RW+\mathcal{A}_Z C_RW+\mathcal{H}\nabla_Z C_RW)   \\
           & = -(\phi_R \mathcal{V}\nabla_Z B_RW+\omega_R\mathcal{V}\nabla_Z B_RW+B_R\mathcal{A}_Z B_RW
             +C_R\mathcal{A}_Z B_RW+\phi_R \mathcal{A}_Z C_RW   \\
           &  \ \ \ +\omega_R \mathcal{A}_Z C_RW  +B_R\mathcal{H}\nabla_Z C_RW+C_R\mathcal{H}\nabla_Z C_RW).
\end{align*}
Thus,
$$
\nabla_Z W\in \Gamma((\ker F_*)^{\perp}) \Leftrightarrow
\phi_R(\mathcal{V}{\nabla}_Z B_RW+\mathcal{A}_Z C_RW)+B_R(\mathcal{A}_Z
B_RW+\mathcal{H}\nabla_Z C_RW) = 0.
$$
Hence, we obtain
$$
a) \Leftrightarrow b), \quad a) \Leftrightarrow c), \quad a)
\Leftrightarrow d).
$$
Therefore, the result follows.
\end{proof}

\begin{theorem}
Let $F$ be a h-semi-slant Riemannian map from a
hyperk\"{a}hler manifold $(M,I,J,K,g_M)$ to a Riemannian
manifold $(N, g_N)$ such that $(I,J,K)$ is a h-semi-slant
basis. Then the following conditions are equivalent:

a) the fibers of $F$ are locally Riemannian product manifolds of the leaves of $\mathcal{D}_1$ and $\mathcal{D}_2$

b)
$$
Q_I(\phi_I\widehat{\nabla}_U \phi_I V+B_I\mathcal{T}_U \phi_I V) =
0 \ \text{and} \ \omega_I\widehat{\nabla}_U \phi_I
V+C_I\mathcal{T}_U \phi_I V = 0
$$
for $U,V\in \Gamma(\mathcal{D}_1)$,

\begin{align*}
  &P_I(\phi_I(\widehat{\nabla}_X \phi_I Y+\mathcal{T}_X \omega_I Y)+B_I(\mathcal{T}_X \phi_I Y
    +\mathcal{H}\nabla_X \omega_I Y)) = 0     \\
  &\omega_I(\widehat{\nabla}_X \phi_I Y+\mathcal{T}_X \omega_I Y)+C_I(\mathcal{T}_X \phi_I Y
    +\mathcal{H}\nabla_X \omega_I Y) = 0
\end{align*}
for $X,Y\in \Gamma(\mathcal{D}_2)$.

c)
$$
Q_J(\phi_J\widehat{\nabla}_U \phi_J V+B_J\mathcal{T}_U \phi_J V) =
0 \ \text{and} \ \omega_J\widehat{\nabla}_U \phi_J
V+C_J\mathcal{T}_U \phi_J V = 0
$$
for $U,V\in \Gamma(\mathcal{D}_1)$,

\begin{align*}
  &P_J(\phi_J(\widehat{\nabla}_X \phi_J Y+\mathcal{T}_X \omega_J Y)+B_J(\mathcal{T}_X \phi_J Y
    +\mathcal{H}\nabla_X \omega_J Y)) = 0     \\
  &\omega_J(\widehat{\nabla}_X \phi_J Y+\mathcal{T}_X \omega_J Y)+C_J(\mathcal{T}_X \phi_J Y
    +\mathcal{H}\nabla_X \omega_J Y) = 0
\end{align*}
for $X,Y\in \Gamma(\mathcal{D}_2)$.

d)
$$
Q_K(\phi_K\widehat{\nabla}_U \phi_K V+B_K\mathcal{T}_U \phi_K V) =
0 \ \text{and} \ \omega_K\widehat{\nabla}_U \phi_K
V+C_K\mathcal{T}_U \phi_K V = 0
$$
for $U,V\in \Gamma(\mathcal{D}_1)$,

\begin{align*}
  &P_K(\phi_K(\widehat{\nabla}_X \phi_K Y+\mathcal{T}_X \omega_K Y)+B_K(\mathcal{T}_X \phi_K Y
    +\mathcal{H}\nabla_X \omega_K Y)) = 0     \\
  &\omega_K(\widehat{\nabla}_X \phi_K Y+\mathcal{T}_X \omega_K Y)+C_K(\mathcal{T}_X \phi_K Y
    +\mathcal{H}\nabla_X \omega_K Y) = 0
\end{align*}
for $X,Y\in \Gamma(\mathcal{D}_2)$.
\end{theorem}

\begin{proof}
Given $R\in \{ I,J,K \}$, for $U,V\in \Gamma(\mathcal{D}_1)$, we get
\begin{align*}
\nabla_U V&= -J\nabla_U JV= -J(\widehat{\nabla}_U \phi V+\mathcal{T}_U \phi V)   \\
          &= -(\phi \widehat{\nabla}_U \phi V+\omega\widehat{\nabla}_U \phi V+B\mathcal{T}_U \phi V
             +C\mathcal{T}_U \phi V).
\end{align*}
Thus,
$$
\nabla_U V\in \Gamma(\mathcal{D}_1) \Leftrightarrow Q(\phi
\widehat{\nabla}_U \phi V+B\mathcal{T}_U \phi V) = 0 \ \text{and} \
\omega\widehat{\nabla}_U \phi V+C\mathcal{T}_U \phi V = 0.
$$
For $X,Y\in \Gamma(\mathcal{D}_2)$, we have
\begin{align*}
\nabla_X Y&= -R\nabla_X RY= -R(\widehat{\nabla}_X
\phi_R Y+\mathcal{T}_X \phi_R Y+\mathcal{T}_X \omega_R Y+\mathcal{H}\nabla_X \omega_R Y)   \\
          &= -(\phi_R \widehat{\nabla}_X \phi_R Y+\omega_R\widehat{\nabla}_X \phi_R Y+B_R\mathcal{T}_X \phi_R Y
             +C_R\mathcal{T}_X \phi_R Y+\phi_R \mathcal{T}_X \omega_R Y   \\
          & \ \ \ +\omega_R \mathcal{T}_X \omega_R Y  +B_R\mathcal{H}\nabla_X \omega_R Y+
          C_R\mathcal{H}\nabla_X \omega_R Y).
\end{align*}
Thus,

$\nabla_X Y\in \Gamma(\mathcal{D}_2) \Leftrightarrow$
$$
  \begin{cases}
     P_R(\phi_R(\widehat{\nabla}_X \phi_R Y+\mathcal{T}_X \omega_R Y)+B_R(\mathcal{T}_X \phi_R Y
    +\mathcal{H}\nabla_X \omega_R Y)) = 0,& \\
     \omega_R(\widehat{\nabla}_X \phi_R Y+\mathcal{T}_X \omega_R Y)+C_R(\mathcal{T}_X \phi_R Y
    +\mathcal{H}\nabla_X \omega_R Y) = 0.&
  \end{cases}
$$
Hence, we have
$$
a) \Leftrightarrow b), \quad a) \Leftrightarrow c), \quad a)
\Leftrightarrow d).
$$
Therefore, we obtain the result.
\end{proof}

\section{Examples}\label{exam}

Note that given an Euclidean space $\mathbb{R}^{4m}$ with
coordinates $(x_1,x_2,\cdots,x_{4m})$, we can canonically choose
complex structures $I, J, K$ on $\mathbb{R}^{4m}$ as follows:
\begin{align*}
  &I(\tfrac{\partial}{\partial x_{4k+1}})=\tfrac{\partial}{\partial x_{4k+2}},
  I(\tfrac{\partial}{\partial x_{4k+2}})=-\tfrac{\partial}{\partial x_{4k+1}},
  I(\tfrac{\partial}{\partial x_{4k+3}})=\tfrac{\partial}{\partial x_{4k+4}},
  I(\tfrac{\partial}{\partial x_{4k+4}})=-\tfrac{\partial}{\partial x_{4k+3}},     \\
  &J(\tfrac{\partial}{\partial x_{4k+1}})=\tfrac{\partial}{\partial x_{4k+3}},
  J(\tfrac{\partial}{\partial x_{4k+2}})=-\tfrac{\partial}{\partial x_{4k+4}},
  J(\tfrac{\partial}{\partial x_{4k+3}})=-\tfrac{\partial}{\partial x_{4k+1}},
  J(\tfrac{\partial}{\partial x_{4k+4}})=\tfrac{\partial}{\partial x_{4k+2}},    \\
  &K(\tfrac{\partial}{\partial x_{4k+1}})=\tfrac{\partial}{\partial x_{4k+4}},
  K(\tfrac{\partial}{\partial x_{4k+2}})=\tfrac{\partial}{\partial x_{4k+3}},
  K(\tfrac{\partial}{\partial x_{4k+3}})=-\tfrac{\partial}{\partial x_{4k+2}},
  K(\tfrac{\partial}{\partial x_{4k+4}})=-\tfrac{\partial}{\partial x_{4k+1}}
\end{align*}
for $k\in \{ 0,1,\cdots,m-1 \}$.
Then it is easy to check that $(I,J,K,<\ ,\ >)$ is a hyperk\"{a}hler structure on $\mathbb{R}^{4m}$,
where $<\ ,\ >$ denotes the Euclidean metric on $\mathbb{R}^{4m}$.
Throughout this section, we will use these notations.

\begin{example}
Let $F$ be an almost h-slant submersion from an almost
quaternionic Hermitian manifold $(M, E, g_M)$ onto a Riemannian
manifold $(N, g_N )$. Then the map $F : (M, E, g_M) \mapsto (N,
g_N )$ is a h-semi-slant Riemannian map with $\mathcal{D}_2 = \ker
F_*$ \cite{P}.
\end{example}

\begin{example}
Let $F$ be an almost h-semi-slant submersion from an almost
quaternionic Hermitian manifold $(M, E, g_M)$ onto a Riemannian
manifold $(N, g_N )$. Then the map $F : (M, E, g_M) \mapsto (N,
g_N )$ is an almost h-semi-slant Riemannian map \cite{P3}.
\end{example}

\begin{example}
Let $(M,E,g)$ be an almost quaternionic Hermitian manifold. Let
$\pi : TM \mapsto M$ be the natural projection. Then the map $\pi$
is a strictly h-semi-slant Riemannian map such that $\mathcal{D}_1 =
\ker \pi_*$ and the strictly h-semi-slant angle $\theta=0$
\cite{IMV}.
\end{example}

\begin{example}
Let $(M,E_M,g_M)$ and $(N,E_N,g_N)$ be almost quaternionic
Hermitian manifolds. Let $F : M \mapsto N$ be a quaternionic
submersion. Then the map $F$ is a strictly h-semi-slant Riemannian map
such that $\mathcal{D}_1 = \ker F_*$ and the strictly h-semi-slant
angle $\theta=0$ \cite{IMV}.
\end{example}

\begin{example}
Define a map $\displaystyle{F : \mathbb{R}^8 \mapsto
\mathbb{R}^4}$ by
$$
F(x_1,\cdots,x_8)=(x_2, x_1\sin \alpha-x_3\cos \alpha, 2012, x_4),
$$
where $\alpha$ is constant. Then the map $F$ is a strictly
h-semi-slant Riemannian map such that
$$
\mathcal{D}_1 = <\frac{\partial}{\partial x_5},
\frac{\partial}{\partial x_6},\frac{\partial}{\partial
x_7},\frac{\partial}{\partial x_8}> \ \text{and} \ \mathcal{D}_2 =
<\cos \alpha\frac{\partial}{\partial x_1}+\sin
\alpha\frac{\partial}{\partial x_3}>
$$
with the strictly h-semi-slant angle $\theta=\frac{\pi}{2}$.
\end{example}

\begin{example}
Let $(M,E,g_M)$ be a $4m-$dimensional almost quaternionic Hermitian manifold
and $(N,g_N)$ a $(4m-1)-$dimensional Riemannian manifold. Let
$F : (M,E,g_M) \mapsto  (N,g_N)$ be a Riemannian map with $\rank F = 4m-1$.
Then the map $F$ is a strictly h-semi-slant Riemannian map such that
$\mathcal{D}_2 = \ker F_*$ and the strictly h-semi-slant angle $\theta=\frac{\pi}{2}$.
\end{example}

\begin{example}
Define a map $\displaystyle{F : \mathbb{R}^{12} \mapsto
\mathbb{R}^5}$ by
$$
F(x_1,\cdots,x_{12})=(x_{6}, \frac{x_1-x_3}{\sqrt{2}}, c,  x_4,
\frac{x_5-x_{7}}{\sqrt{2}}),
$$
where $c$ is constant.
Then the map $F$ is a h-semi-slant Riemannian map such that
$$
\mathcal{D}_1 = <\frac{\partial}{\partial x_9},
\frac{\partial}{\partial x_{10}},\frac{\partial}{\partial x_{11}},
\frac{\partial}{\partial x_{12}}> \ \text{and} \ \mathcal{D}_2 =
<\frac{\partial}{\partial x_2},\frac{\partial}{\partial
x_{8}},\frac{\partial}{\partial x_1}+\frac{\partial}{\partial
x_3}, \frac{\partial}{\partial x_5}+\frac{\partial}{\partial
x_{7}}>
$$
with the h-semi-slant angles $\{ \theta_I=
\frac{\pi}{4},\theta_J=\frac{\pi}{2},\theta_K=\frac{\pi}{4} \}$.
\end{example}

\begin{example}
Define a map $F : \mathbb{R}^{12} \mapsto \mathbb{R}^7$ by
$$
F(x_1,\cdots,x_{12})=(x_5\cos \alpha-x_7\sin \alpha, \gamma, x_6\sin
\beta-x_8\cos \beta, x_9, x_{11}, x_{12}, x_{10}),
$$
where $\alpha$, $\beta$, and $\gamma$ are constant. Then the map $F$ is a
h-semi-slant Riemannian map such that
$$
\mathcal{D}_1 = <\frac{\partial}{\partial x_1},
\frac{\partial}{\partial x_{2}},\frac{\partial}{\partial
x_{3}},\frac{\partial}{\partial x_{4}}>
$$
and
$$
\mathcal{D}_2 = <\sin \alpha\frac{\partial}{\partial x_5}+\cos
\alpha\frac{\partial}{\partial x_7}, \cos
\beta\frac{\partial}{\partial x_6}+\sin
\beta\frac{\partial}{\partial x_8}>
$$
with the h-semi-slant angles $\{
\theta_I,\theta_J=\frac{\pi}{2},\theta_K \}$ such that $\cos
\theta_I=|\sin (\alpha+\beta)|$ and $\cos \theta_K=|\cos
(\alpha+\beta)|$.
\end{example}

\begin{example}
Define a map $F : \mathbb{R}^{12} \mapsto \mathbb{R}^7$ by
$$
F(x_1,\cdots,x_{12})=(x_3, x_4, 0, x_7, x_5, x_6, x_8).
$$
Then the map $F$ is an almost h-semi-slant Riemannian map  such that
\begin{align*}
&\mathcal{D}_1^I = <\frac{\partial}{\partial x_1},
   \frac{\partial}{\partial x_2},\frac{\partial}{\partial x_9},
  \frac{\partial}{\partial x_{10}},\frac{\partial}{\partial x_{11}},
  \frac{\partial}{\partial x_{12}}>,     \\
&\mathcal{D}_1^J =\mathcal{D}_1^K = <\frac{\partial}{\partial
   x_9}, \frac{\partial}{\partial x_{10}},\frac{\partial}{\partial
   x_{11}},\frac{\partial}{\partial x_{12}}>,     \\
&\mathcal{D}_2^I = 0, \quad \mathcal{D}_2^J =\mathcal{D}_2^K =
<\frac{\partial}{\partial x_1},
   \frac{\partial}{\partial x_2}>.
\end{align*}
with the almost h-semi-slant angles $\{
\theta_I=0,\theta_J=\frac{\pi}{2},\theta_K=\frac{\pi}{2} \}$.
\end{example}

\begin{example}
Define a map $F : \mathbb{R}^{12} \mapsto \mathbb{R}^6$ by
$$
F(x_1,\cdots,x_{12})=(x_2, x_5, \alpha, x_1, \beta, x_7),
$$
where $\alpha$ and $\beta$ are constant.
Then the map $F$ is an almost h-semi-slant Riemannian map such that
\begin{align*}
&\mathcal{D}_1^I = <\frac{\partial}{\partial x_3},
   \frac{\partial}{\partial x_4},\frac{\partial}{\partial x_9},
  \frac{\partial}{\partial x_{10}},\frac{\partial}{\partial x_{11}},
  \frac{\partial}{\partial x_{12}}>,     \\
&\mathcal{D}_1^J = <\frac{\partial}{\partial x_6},
   \frac{\partial}{\partial x_8},\frac{\partial}{\partial x_9},
  \frac{\partial}{\partial x_{10}},\frac{\partial}{\partial x_{11}},
  \frac{\partial}{\partial x_{12}}>,     \\
&\mathcal{D}_1^K = <\frac{\partial}{\partial x_9},
  \frac{\partial}{\partial x_{10}},\frac{\partial}{\partial x_{11}},
  \frac{\partial}{\partial x_{12}}>,     \\
&\mathcal{D}_2^I =<\frac{\partial}{\partial
x_6},\frac{\partial}{\partial x_8}>, \quad \mathcal{D}_2^J =
<\frac{\partial}{\partial x_3}, \frac{\partial}{\partial x_4}>,   \\
&\mathcal{D}_2^K = <\frac{\partial}{\partial
x_3},\frac{\partial}{\partial x_4},\frac{\partial}{\partial
x_6},\frac{\partial}{\partial x_8}>
\end{align*}
with the almost h-semi-slant angles $\{
\theta_I=\frac{\pi}{2},\theta_J=\frac{\pi}{2},\theta_K=\frac{\pi}{2}
\}$.
\end{example}

\begin{example}
Let $\widetilde{F}$ be a h-semi-slant Riemannian map from an almost
quaternionic Hermitian manifold $(M_1, E_1, g_{M_1})$ to a
Riemannian manifold $(N, g_N )$ with $\mathcal{D}_2 = \ker
\widetilde{F}_*$. Let $(M_2, E_2, g_{M_2})$ be an
almost quaternionic Hermitian manifold. Denote by $(M,E,g_M)$ the warped product of $(M_1, E_1, g_{M_1})$ and
$(M_2, E_2, g_{M_2})$ by a positive function $g$ on $M_1$ \cite {FIP}, where $E=E_1\times E_2$.

Define a map $F : (M,E,g_M) \mapsto (N,g_N)$ by
$$
F(x,y) = \widetilde{F}(x) \quad \text{for} \ x\in M_1 \ \text{and} \
y\in M_2.
$$
Then the map $F$ is a h-semi-slant Riemannian map such that
$$
\mathcal{D}_1 = TM_2 \ \text{and} \ \mathcal{D}_2 = \ker
\widetilde{F}_*
$$
with the h-semi-slant angles $\{ \theta_I,\theta_J,\theta_K \}$,
where $\{ I,J,K \}$ is a h-slant basis for the map
$\widetilde{F}$ with the h-semi-slant angles $\{ \theta_I,\theta_J,\theta_K
\}$.

Note that as a generalization of an almost h-slant submersion \cite{P}, we call the map $\widetilde{F}$ an {\em almost h-slant Riemannian map}.
\end{example}



\end{document}